\input ./style/arxiv-vmsta.cfg
\documentclass[numbers,compress,v1.0.1]{vmsta}
\volume{3}
\issue{2}
\pubyear{2016}
\firstpage{181}
\lastpage{190}
\doi{10.15559/16-VMSTA59}

\startlocaldefs

\allowdisplaybreaks

\newtheorem{corollary}{Corollary}
\newtheorem{proposition}{Proposition}
\newcommand{\E}{\mathbf{E}}
\newcommand{\Var}{\mathbf{Var}}
\newcommand{\Corr}{\mathbf{Corr}}
\newcommand{\Cov}{\mathbf{Cov}}
\newcommand{\R}{\mathbb{R}}
\newcommand{\pr}{\mathbf{P}}

\endlocaldefs

\begin{document}

\begin{frontmatter}

\title{Simulation paradoxes related to a fractional Brownian motion
with small Hurst index}

\author{\inits{V.}\fnm{Vitalii}\snm{Makogin}}\email{makoginv@ukr.net}
\address{Department of Probability Theory, Statistics and Actuarial Mathematics,
Taras Shevchenko National University of Kyiv,
64, Volodymyrska St., 01601 Kyiv, Ukraine}

\markboth{V. Makogin}{Simulation paradoxes related to a fractional
Brownian motion with small Hurst index}

\begin{abstract}
We consider the simulation of sample paths of a fractional Brownian
motion with small values of the Hurst index and estimate the behavior
of the expected maximum. We prove that, for each fixed $N$, the error
of approximation $\E\max_{t \in[0,1]}B^H(t) - \E\max_{i =
\overline{1,N}} B^H (i/N )$ grows rapidly to $\infty$ as the
Hurst index tends to 0.
\end{abstract}

\begin{keyword}
\xch{Fractional}{fractional} Brownian motion\sep
Monte Carlo simulations \sep
expected maximum\sep
discrete approximation
\MSC[2010]
65C50 \sep60G22
\end{keyword}
\received{31 May 2016}
\revised{19 June 2016}
\accepted{20 June 2016}
\publishedonline{4 July 2016}
\end{frontmatter}

\section{Introduction}\label{}

A fractional Brownian motion $\{B^H(t),t\geq0\}$ is a centered
Gaussian stochastic process with covariance function
\[
\E\bigl[B^H(t)B^H (u)\bigr] = \frac{1}2
\bigl(t^{2H} +u^{2H} -|t - u|^{2H}\bigr),\quad t,u\geq0,
\]
where $H\in(0,1)$ is the Hurst index. The fractional Brownian motion
is a~self-similar process with index $H$, that is, for any $a > 0$,
\[
\bigl\{B^H(t),t\geq0\bigr\}\stackrel{d} {=}\bigl
\{a^{-H}B^H(at), t \geq0\bigr\},
\]
where $\stackrel{d}{=}$ means the equality of finite-dimensional distributions.

Due to self-similarity, we have that, for all $T>0$,
\[
\bigl\{B^H(t),t\in[0,T]\bigr\}\stackrel{d} {=}\bigl
\{T^{H}B^H\bigl(T^{-1}t\bigr), t \in[0,T]\bigr\}=
\bigl\{ T^{H}B^H(s), s \in[0,1]\bigr\}.
\]
Based on such a invariance of distributions, it is appropriate to
investigate the properties of the fractional Brownian motion only over
the time interval $[0,1]$.

In this paper, we consider the behavior of the maximum functional\break
$\max_{t \in[0,1]}B^H(t)$ with small values of Hurst index.

It should be noted that the fractional Brownian motion process with
$H=1/2$ is the Wiener process $\{W(t),t\geq0\}$. The distribution of
$\max_{t \in[0,1]}W(t)$ is known. Namely,
\[
\pr \Bigl(\max_{t \in[0,1]}W(t)\leq x \Bigr)=\sqrt{\frac{2}{\pi}}
\int_{0}^x e^{-y^2/2}dy,\quad x\ge0,
\]
and, therefore,
\[
\E \Bigl[\max_{t \in[0,1]}W(t) \Bigr]=\sqrt{\frac{2}{\pi}}.
\]

Many papers are devoted to the distribution of the maximum functional
of the fractional Brownian motion, where usually asymptotic properties
for large values of time horizon $T$ are considered. For example,
Molchan \cite{Mol} has found an asymptotic behavior of small-ball
probabilities for the maximum of the fractional Brownian motion.
Talagrand \cite{Tal} obtained lower bounds for the expected maximum of
the fractional Brownian motion. In several works, the distribution of
the maximum is investigated when the Hurst index $H$ is close to $1/2$.
In particular, this case was considered by Sinai \cite{Sinai} and
recently by Delorme and Weise \cite{Del}.

Currently, an analytical expression for the distribution of the maximum
of the functional Brownian motion remains unknown. Moreover, the exact
value of the expectation of such a functional is unknown too.

From the paper of Borovkov et al.\ \cite{Borovkov} we know the
following bounds:
\begin{equation}
\label{maxbounds} \frac{1}{2\sqrt{H \pi e \ln2}} \leq\E\max_{t \in
[0,1]}B^H(t)
< \frac{16.3}{\sqrt{H}}.
\end{equation}
On the other hand, we may get an approximate value of the expected
maximum using Monte Carlo simulations. That is, for sufficiently large
$N $,
\begin{equation}
\label{appr} \E\max_{t \in[0,1]}B^H(t) \approx\E\max
_{i = \overline{1,N}} B^H(i/N).
\end{equation}
The authors of \cite{Borovkov} obtain an upper bound for the error
$\varDelta_N$ of approximation~\eqref{appr}. Namely, for $N\geq2^{1/H}$,
\begin{align}
\label{apprbounds1} 0\leq\varDelta_N:={}& \E\max_{t \in[0,1]}B^H(t)
- \E\max_{i = \overline{1,N}} B^H (i/N )
\\[4pt]
\label{apprbounds2}\leq{}&\frac{2 \sqrt{\ln N}}{N^H} \biggl(1+\frac{4}{N^H}+
\frac{0.0074}{(\ln N)^{3/2}} \biggr).
\end{align}

The implementation of approximation \eqref{appr} has technical
limitations. Due to modern computer capabilities, we assume that $N
\leq2^{20} \approx10^6$. Under such conditions, inequality \eqref
{apprbounds2} is true when $H \geq0.05$, and $\varDelta_N < 11.18$.

In this article, we make Monte Carlo simulations and estimate \linebreak
$\E\max_{i = \overline{1,N}} B^H(i/N)$. Also, we investigate the
behavior of $\varDelta_N$ with small values of the Hurst index $H$ and
show that, for a fixed $N$, the approximation error $\varDelta_N\to+\infty
$ as $H\to0$. For the rate of this convergence, when $N=2^{20}$, we
prove the inequality $\varDelta_N>c_1 H^{-1/2} - c_2,\ H\in(0,1)$, where
the constants $c_1=0.2055$ and $c_2=3.4452$ are calculated numerically.
Thus, when the values of $H$ are small, approximation \eqref{appr} is
not appropriate for evaluation of $\E\max_{t \in[0,1]}B^H(t)$.

The article is organized as follows. The first section presents the
methodology of computing. The second section presents the results of
computing of the expected maximum of the fractional Brownian motion.
In the third section, we obtain a lower bound for the error $ \varDelta_N$
and calculate the constants $c_1 $ and~$c_2 $.

\section{Methods of approximate calculations}
\subsection{Simulation of a vector $(B^H(i/N))_{1\leq i \leq N}$}

Let us consider briefly the method proposed by Wood and Chan \cite
{W-Ch}. Let $G$ be the autocovariance matrix of $(B^H(1/N),\ldots
,B^H(N/N))$. Embed $G$ in a~circulant $m\times m$ matrix $C$ given by
\begin{equation*}
C= %
\begin{pmatrix}
c_0 & c_1 & \cdots& c_{m-1}\\
c_{m-1} & c_0 & \cdots& c_{m-2}\\
\vdots& \vdots& \ddots& \vdots\\
c_1 & c_2 & \cdots& c_{0}
\end{pmatrix} %
,
\end{equation*}
where
\begin{equation*}
c_j= %
\begin{cases}
\frac{1}{N^{2H}} (|j-1|^{2H}-2 j^{2H}+(j+1)^{2H} ), & 0\le j\le\frac
{m}{2}, \\[3pt]
\frac{1}{N^{2H}} ((m-j-1)^{2H}-2(m-j)^{2H}+(m-j+1)^{2H} ), & \frac
{m}{2}< j\le m-1.
\end{cases} %
\end{equation*}
\begin{proposition}
Let $m=2^{1+\nu}$, where
$ 2^\nu$ is the minimum power of $2$ not less than~$N$. Then the
matrix $C$ allows a representation $ C = QJQ^T$, where $J$ is a
diagonal matrix of eigenvalues of the matrix $C$, and $Q$ is the
unitary matrix with elements
\begin{equation*}
(Q)_{j,k}=\frac{1}{\sqrt{m}}c_j\exp \biggl(-2 i \pi
\frac{jk}{m} \biggr),\quad j,k=\overline{0,m-1}.
\end{equation*}
The eigenvalues $\lambda_k$ of the matrix $C$ are equal to
\begin{equation*}
\lambda_k=\sum_{j=0}^{m-1}\exp
\biggl(2 i \pi\frac{jk}{m} \biggr),\quad k=\overline{0,m-1}.
\end{equation*}
\end{proposition}

Since $Q$ is unitary, we can set $Y=Q J^{1/2}Q^{T} Z$, where $Z\sim
N(0,I_m)$. Therefore, we get $Y\sim N(0,C)$. Thus, the distributions of
the vectors $(Y_0$, $Y_0+Y_1,\ldots,Y_0+\cdots+Y_{N-1})$ and
$(B^H(1/N),\ldots,B^H(N/N))$ coincide.

The method of Wood and Chan is exact and has complexity $ O (N \log
(N))$. A more detailed description of the algorithm, a comparison with
other methods of simulation of the fractional Brownian motion, and a
program code are contained in the paper \cite{Coe}. For reasons of
optimization of calculations, simulations in the present paper are made
by the method of Wood and Chan.

The estimate of the mean value $\E\max_{i = \overline{1,N}} B^H(i/N)$
is a sample mean over the sample of size $n$. That is why
the total complexity of the algorithm is $O(n N\log(N))$.

\subsection{Clark's method}

Instead of generating samples and computing sample means, there exists
a~method of Clark \cite{clark} for approximating the expected maximum.

Due to this method, the first four moments of the random variable $\max
\{\xi_1, \ldots,\allowbreak\xi_N\}$, where $(\xi_1,\ldots,\xi_N)$ is a
Gaussian vector, are calculated approximately. Since the fractional
Brownian motion is a Gaussian process, we put $(\xi _1,\ldots,\xi_N)
=\break(B^H(1/N),\ldots,B^H(N/N))$ and apply Clark's method for
approximate computing of $\E\max_{i=\overline{1,N}}B^H (i/N )$.

Let us illustrate the basic idea of Clark's method of calculating
\linebreak$\E\max\{\xi, \eta, \tau\}$, where $ \xi, \eta, \tau$ are
Gaussian distributed.

\begin{proposition}
Let $\xi,\eta,\tau$ be Gaussian random variables. Put $a=\Var(\xi)+\Var
(\eta)-\Cov(\xi,\eta)$ and let $a>0$.
Denote $\alpha:=(\E\xi-\E\eta)/a$. Then we have
\begin{align}
\E\max\{\xi,\eta\}&=\varPhi(\alpha) \E\xi+ \varPhi (-\alpha) \E
\eta+ a \varphi(\alpha);
\nonumber\\
\E \bigl(\max\{\xi,\eta\} \bigr)^2&=\varPhi(\alpha) \E
\xi^2 + \varPhi(-\alpha) \E\eta^2 + a \varphi(\alpha) ( \E
\xi+\E\eta),\label{Clark-exp}
\end{align}
where
$\varphi(x)=\frac{1}{\sqrt{2\pi}}\exp(-\frac{x^2}{2} )$ and $\varPhi
(x)=\int_{-\infty}^x \phi(t)dt$.
\end{proposition}

So, the exact value of $\E\max\{\xi,\eta\}$ is obtained from the
previous proposition.
\begin{proposition}
Let $\xi,\eta,\tau$ be Gaussian random variables. Let $\Corr(\tau,\xi)$
and\break $\Corr(\tau,\eta)$ be known. Then
\begin{equation*}
\Corr \bigl(\tau,\max\{\xi,\eta\} \bigr)=\frac{\sqrt{\Var(\xi)}\Corr(\tau
,\xi)\varPhi(\alpha)+\sqrt{\Var(\eta)}\Corr(\tau,\eta)\varPhi(-\alpha
)}{\sqrt{\E(\max\{\xi,\eta\})^2-(\E\max\{\xi,\eta\})^2}}.
\end{equation*}
\end{proposition}

For approximate computing $\E\max\{\xi,\eta,\tau\}=\E\max\{\tau,\max\{
\xi,\eta\}\}$, we assume that $\max\{\xi,\eta\}$ has a Gaussian
distribution. In fact, this is not true, but it allows us to apply
formula \eqref{Clark-exp} for random variables $\tau$ and $\max\{\xi
,\eta\}$.
Thus, iteratively, we can calculate the approximate mean for any finite
number of Gaussian random variables.

\section{Computing the expected maximum}
In this section, we present results of approximate computing 
$\E\max_{i=\overline{1,N}}B^H (i/N )$
by generating random samples and applying Clark's method. Also, we
compare the computational results obtained by these two methods.

The values of the Hurst index are taken from the set $\{
10^{-4}(1+4i),\, i=\overline{0,24}\}\cup\{10^{-2}i,\,i=\overline{1,9}\}
$. The values of $N$ are chosen from the set $\{2^j,j=\overline{8,19}\}$.
The values of $(B^H(1/N),B^H(2/N)$, $\ldots, B^H(N/N))$ are simulated
by the method of Wood and Chan for each pair $N,H$ with the sample size
$n=1000$. For each element in the sample, we calculate the following
functionals:
\begin{align}
\label{maxf} &\max_{i=\overline{1,N}}B^H (i/N ),
\\
\label{intf} &\frac{1}{N}\sum_{i=1}^N
B^H (i/N ).
\end{align}

\subsection{Approximation error of $ \frac{1}{N}\sum_{i=0}^N B^H (i/N)$}

We compute the sample mean and variance of \eqref{intf}.
The values of theoretical moments of \eqref{intf} are known:
\[
\E\frac{1}{N}\sum_{i=0}^N
B^H(i/N)=0,
\]
\[
\quad\E \Biggl(\frac{1}{N}\sum_{i=0}^N
B^H(i/N) \Biggr)^2=\frac{1}{N^{2H+2}}\sum
_{i=1}^{N}{i^{2H+1}}\to\frac{1}{(2H+2)},\quad N
\to\infty.
\]

\begin{figure}
\includegraphics{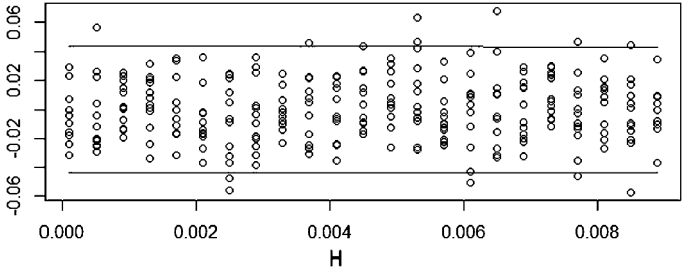}
\caption{Sample means of $\frac{1}{N}\sum_{i=0}^N B^H(i/N)$}\label{integrE}
\end{figure}
\begin{figure}
\includegraphics{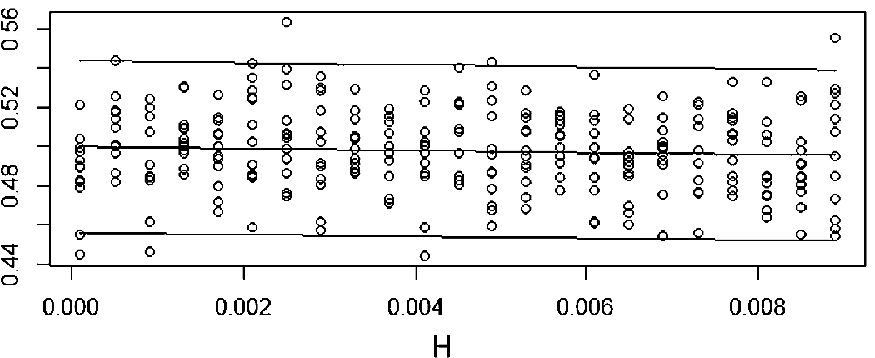}
\caption{Sample variances of $\frac{1}{N}\sum_{i=0}^N B^H(i/N)$}\label{integrV}
\end{figure}

The sample moments of \eqref{intf} when $H=\{10^{-4}(1+4i),\,i=\overline
{0,24}\}$ and $N=\{2^j,\,j=\overline{8,19}\}$ are presented in
Figs.~\ref{integrE} and \ref{integrV}.
In the figures, the lines indicate the theoretical moments and
confidence intervals corresponding to the reliability of 95\%.
The data confirm the correctness of calculations of \eqref{intf} with
the reliability of 95\% even for small values of $H$.

\subsection{Computing functional $\max_{i = \overline{1, N}} B^H (i/N) $}

For each pair $ N, H$, we obtain the sample of values of the maximum
functional~\eqref{maxf}
with sample size $1000$. For some values of $ H $, the sample means
and approximate values of the expected maximum, obtained by Clark's
method, are presented in Table~\ref{tableWC}.

\begin{table}
\caption{\xch{The approximate values of the expected maximum}{}} \label{tableWC}
\renewcommand{\arraystretch}{1.05}
\tabcolsep=4pt
\begin{tabular*}{\textwidth}{@{\extracolsep{\fill}}l@{}cccc@{}cccc@{}}
\hline
&\multicolumn{4}{c}{Sample means of \eqref{maxf}}  &  \multicolumn{4}{c}{Values due to Clark's method}\\
\cline{2-5}
\cline{6-9}
$N\diagdown H$ & 0.0900 & 0.0100 & 0.0013 & 0.0001 &  0.0900 & 0.0100 & 0.0013 & 0.0001 \\
\hline
$2^8$ & 1.7017 & 2.0019 & 1.9897 & 1.9769 &           1.1738 & 1.8691 & 1.9696 & 1.9839 \\
$2^9$ & 1.7693 & 2.0875 & 2.1602 & 2.1360 &           1.1903 & 1.9991 & 2.1194 & 2.1366 \\
$2^{10}$ & 1.9487 & 2.2504 & 2.3047 & 2.2854 &        1.1971 & 2.1193 & 2.2604 & 2.2806 \\
$2^{11}$ & 2.0138 & 2.4203 & 2.4446 & 2.4184 &        1.1966 & 2.2310 & 2.3939 & 2.4174 \\
$2^{12}$ & 2.0886 & 2.5086 & 2.5948 & 2.5334 &        1.1910 & 2.3351 & 2.5208 & 2.5476 \\
$2^{13}$ & 2.1938 & 2.6396 & 2.6934 & 2.6885 &        1.1822 & 2.4327 & 2.6420 & 2.6723 \\
$2^{14}$ & 2.2591 & 2.7612 & 2.7829 & 2.7940 &        1.1714 & 2.5242 & 2.7579 & 2.7919 \\
$2^{15}$ & 2.3327 & 2.8837 & 2.9452 & 2.9258 &        1.1586 & 2.6104 & 2.8693 & 2.9070 \\
$2^{16}$ & 2.4050 & 2.9973 & 3.0526 & 3.0464 &        1.1436 & 2.6917 & 2.9765 & 3.0181 \\
$2^{17}$ & 2.4620 & 3.0791 & 3.1386 & 3.1121 &        1.1263 & 2.7685 & 3.0798 & 3.1256 \\
$2^{18}$ & 2.5328 & 3.1900 & 3.2102 & 3.2421 &        1.1068 & 2.8412 & 3.1798 & 3.2297 \\
$2^{19}$ & 2.5597 & 3.3481 & 3.3487 & 3.3663 &        1.0855 & 2.9101 & 3.2766 & 3.3307 \\
\hline
\end{tabular*}
\end{table}

Within the data obtained by the different methods, we get that the
approximate values obtained by Clark's algorithm differ from the sample
means at most by 57\xch{.}{,}6\% when $H=0.09$, by 13\xch{.}{,}08\% when $H=0.01$, by
2\xch{.}{,}85\% when $H=0.0013$, and by 1.06\% when $H=0.0001$.
Thus, when $H\leq0.0013$,
the values of the expected maximum, obtained by these completely
different methods, are numerically identical. This indicates that the
sample mean is approximately equal to $\E\max_{i=\overline{1,N}}B^H(i/N)$.

\section{Bounds for the approximation error}
In this section, we find bounds for the error of approximation \eqref
{appr}. As noted before, $\E\max_{t \in[0,1]}B^H(t)$ $\ge(4H \pi e
\ln2)^{-1/2}$.
It is expected that obtained sample means of the maximum functional
\eqref{maxf} also satisfies this constraint.
In Fig.~\ref{boundsPIC}, the sample means and the values of $(4H \pi e
\ln2)^{-1/2}$ are presented.
\begin{figure}[t]
\includegraphics{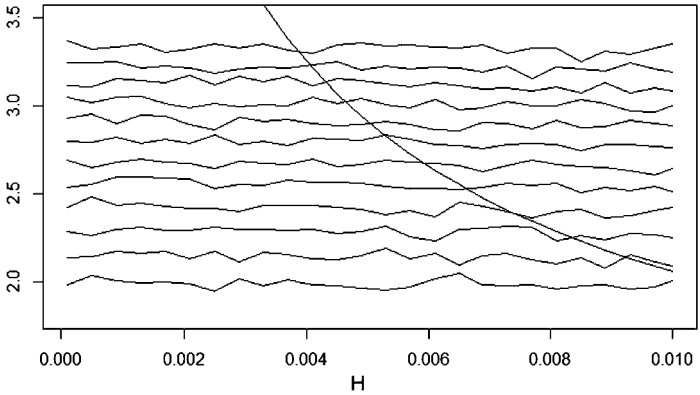}
\caption{Sample means of the maximal functional}
\label{boundsPIC}
\end{figure}
As one can see, the inequality $\E\max_{i =\overline{1,N}}B^H(i/N) \ge
(4H \pi e \ln2)^{-1/2}$ is false for small values of $ H$.

There are two possible explanations of this fact: either there is a
significant error in calculations, or the approximation error $\varDelta
_N$ grows rapidly as $ H \to0$. Let us verify these two explanations.

From \cite[Theorem 4.2]{Borovkov} we get that the expectation of the
maximal functional \eqref{maxf} grows as $H\to0$ and has the limit
\begin{equation}
\label{Eqlim} \lim_{H\to0}\E\max_{i=\overline{1,N}}B^{H}(i/N)=
\frac{1}{\sqrt{2}}\E \Bigl(\max_{i=\overline{1,N}}\xi_i
\Bigr)^+,
\end{equation}
where $\xi_1,\ldots,\xi_N$ are i.i.d. r.v.s, $\xi_1\sim N(0,1)$, and
$x^+:=\max\{0,x\}$.\vadjust{\eject}

Moreover, the rate of convergence in \eqref{Eqlim} is also obtained in
\cite{Borovkov}:
\begin{equation}
\label{Evel} 0 \leq\frac{1}{\sqrt{2}}\E \Bigl(\max_{i=\overline{1,N}}
\xi_i \Bigr)^+ - \E\max_{i=\overline{1,N}}B^{H}(i/N)
\leq1-\frac{1}{N^{2H}}.
\end{equation}
The right-hand side of \eqref{Evel} does not exceed 0.1 when $N=2^{20}$
and $H<0.0038$.

We apply two approaches to calculate $\frac{1}{\sqrt{2}}\E (\max_{i=\overline{1,N}}\xi_i )^+$. The first one is Monte Carlo simulations.

The sample means of $\frac{1}{\sqrt{2}} (\max_{i=\overline{1,N}}\xi_i
)^+$ are presented in Table~\ref{table4} for several sample sizes $n$.

\begin{table}
\caption{\xch{Values of limit \eqref{Eqlim}}{}}
\label{table4}
\tabcolsep=3pt
\renewcommand{\arraystretch}{1.05}
\begin{tabular*}{\textwidth}{@{\extracolsep{\fill}}lccccc@{}c@{}}
\hline
&
\multicolumn{5}{c}{\rule{0pt}{10pt}Sample means of $\frac{1}{\sqrt{2}} (\max_{i=\overline{1,N}}\xi_i )^+$}&
$N \int_{0.5}^{1}\mathrm{erf}^{(-1)}(2z-1) z^{N-1}d z$\\
\noalign{\vspace*{1pt}}\cline{2-6}
$N\diagdown n$ & 1000 & 5000 & 10000 & 15000 & 20000 &\\
\hline
$2^8$ & 1.9908 & 1.9908 & 1.9957 & 1.9965 & 1.9961 & 1.9989\\
$2^9$ & 2.1462 & 2.1506 & 2.1520 & 2.1526 & 2.1525 &2.1524\\
$2^{10}$ & 2.3071 & 2.3033 & 2.3006 & 2.3004 & 2.2994 &2.2969\\
$2^{11}$ & 2.4409 & 2.4362 & 2.4360 & 2.4371 & 2.4351 &2.4337\\
$2^{12}$ & 2.5712 & 2.5657 & 2.5648 & 2.5635 & 2.5643 &2.5640\\
$2^{13}$ & 2.6824 & 2.6847 & 2.6877 & 2.6874 & 2.6867 &2.6887\\
$2^{14}$ & 2.8150 & 2.8066 & 2.8065 & 2.8060 & 2.8078 &2.8082\\
$2^{15}$ & 2.9190 & 2.9259 & 2.9235 & 2.9248 & 2.9244 &2.9232\\
$2^{16}$ & 3.0301 & 3.0372 & 3.0353 & 3.0340 & 3.0348 &3.0343\\
$2^{17}$ & 3.1387 & 3.1372 & 3.1424 & 3.1418 & 3.1414 &3.1417\\
$2^{18}$ & 3.2394 & 3.2456 & 3.2469 & 3.2460 & 3.2461 &3.2458\\
$2^{19}$ & 3.3402 & 3.3442 & 3.3450 & 3.3458 & 3.3460 &3.3469\\
\hline
\end{tabular*}
\end{table}
As we see, with increasing sample size 20 times, the sample means
differ at most by 0.33\% for each $N$. Therefore, to ensure the
accuracy of calculations, it suffices to put $n = 1000$.
Under such conditions, technical resources allow us to calculate the
sample means for larger values of $ N$. In Table~\ref{table5}, the
values of the sample means are presented for $ N =\{2 ^ {20}, 2 ^ {21},
2 ^ {22}, 2 ^ {23}, 2 ^ {24}, 2 ^ {25} \} $ .\vadjust{\eject}

\begin{table}[h]
\caption{\xch{Values of limit \eqref{Eqlim} for $N\geq 2^{20}$}{}}
\label{table5}
\begin{tabular}{lcccccc}
\hline
$N$                                                                    & \rule{0pt}{10pt}$2^{20}$ & $2^{21}$ & $2^{22}$ & $2^{23}$ & $2^{24}$ & $2^{25}$ \\\hline
Sample means of $\frac{1}{\sqrt{2}} (\max_{i=\overline{1,N}}\xi_i )^+$ & \rule{0pt}{10pt} 3.4516  & 3.536    & 3.627    & 3.724    & 3.816    & 4.073    \\
$N \int_{0.5}^{1}\mathrm{erf}^{(-1)}(2z-1) z^{N-1}d z$                 & 3.4452                   & 3.541    & 3.634    & 3.726    & 3.815    & 3.902    \\\hline
\end{tabular}
\end{table}

Instead of generating random samples, we may calculate the value of
\linebreak $\frac{1}{\sqrt{2}}\E (\max_{i=\overline{1,N}}\xi_i )^+$
as an integral.
\begin{proposition}
Let $\xi_1,\ldots,\xi_N$ be i.i.d. r.v.s, $\xi_1\sim N(0,1)$. Then
\begin{align}
\nonumber
\frac{1}{\sqrt{2}}\E \Bigl(\max_{i=\overline{1,N}}
\xi_i \Bigr)^+&= \frac{N}{\sqrt{2}} \int_{1/2}^{1}
\varPhi^{(-1)}(z) z^{N-1}d z
\\
\label{intmax}&=N \int_{1/2}^{1}
\mathrm{erf}^{(-1)}(2z-1) z^{N-1}d z,
\end{align}
where $\varPhi^{(-1)}$ is the inverse function of $\varPhi(x)=\int_{-\infty}^{x}\frac{e^{-y^2/2}}{\sqrt{2\pi}}dy$, $x\in\R$, and $\mathrm
{erf}^{(-1)}$ is the inverse function of the error function $\mathrm{erf}$.
\end{proposition}
\begin{proof}
The proposition follows straightforwardly by quantile transformation.
\end{proof}

We immediately get the following corollary.
\begin{corollary}For any $H\in(0,1)$ and $N\geq1$, we have
\begin{equation}
\label{ErB} \E\max_{i=\overline{1,N}}B^{H}(i/N)\leq N \int
_{1/2}^{1}\mathrm{erf}^{(-1)}(2z-1)
z^{N-1}d z.
\end{equation}
\end{corollary}
The integrand in \eqref{ErB} is not an elementary function, but its
values are tabulated, and there exist methods for its numerical
computing. For the present paper, the integral $N \int_{0.5}^{1}\mathrm
{erf}^{(-1)}(2z-1) z^{N-1}d z$ is calculated numerically, and the
corresponding values are presented in Tables~\ref{table4} and \ref{table5}.
By maintaining the accuracy of calculations, the maximum possible value
of $N$ is $2^{31}$, and the value of the integral reaches 4.390.

The values of $\frac{1}{\sqrt{2}}\E (\max_{i=\overline{1,N}}\xi_i
)^+$, obtained by the two methods, differ at most by 0.44 \% when
$N\leq2^{24}$. When $N=2^{20}$, the absolute error of numerical
computing of \eqref{intmax} is less than $1.3\times10^{-5}$.
Thereafter, for $N=2^{20}$, inequality \eqref{ErB} becomes
\begin{equation}
\label{ErightB} \E\max_{i=\overline{1,N}}B^H (i/N ) \leq 3.4452,
H\in(0,1).
\end{equation}

Let us return to the lower bound for $ \E\max_ {i = \overline{1, N}}
B^H (i / N )$. By Sudakov's inequality \cite{Borovkov,Sudakov} we have
\begin{equation}
\label{EleftB} \E\max_{i=\overline{1,N}}B^H (i/N )\ge\sqrt{
\frac{\ln(N+1)}{ N^{2H}2\pi\ln2}}.
\end{equation}

Moreover, the maximum of the right-hand side of \eqref{EleftB} equals
$(4H \pi e \ln2)^{- 1/2} $ and is reached when $ N = [e ^ {1 / {2H}} ] $.
The values of the lower bound are presented in Table~\ref{values}.

\begin{table}
\caption{Lower bounds}
\label{values}
\begin{tabular*}{\textwidth}{@{\extracolsep{\fill}}clcccccc@{}}
\hline
&& $H$ & 0.5000 & 0.0900 & 0.0100 & 0.0013 & 0.0001 \\
\hline
\rule{0pt}{9pt}$(2\sqrt{H \pi e \ln2})^{-1}$                &          && 0.5811 & 1.3696 & 4.1089               & 11.396                & 41.089                \\
$e^{1/2H}$                                   &          && 2.7183 & 258.67 & 5.18 $\times10^{21}$ & 1.1 $\times 10^{167}$ & $2.97\times10^{2171}$ \\\hline
                                             & $N$      &&        &        &                      &                       &                       \\\hline
                                             & $2^8$    && 0.0705 & 0.6853 & 1.0679               & 1.1207                & 1.1282                \\
                                             & $2^9$    && 0.0529 & 0.6828 & 1.1246               & 1.1873                & 1.1963                \\
                                             & $2^{10}$ && 0.0394 & 0.6761 & 1.1772               & 1.2503                & 1.2608                \\
                                             & $2^{11}$ && 0.0292 & 0.6662 & 1.2260               & 1.3101                & 1.3222                \\
                                             & $2^{12}$ && 0.0216 & 0.6537 & 1.2717               & 1.3671                & 1.3808                \\
$ (\frac{\ln(N+1)}{ N^{2H}2\pi\ln2} )^{1/2}$ & $2^{13}$ && 0.0159 & 0.6393 & 1.3145               & 1.4217                & 1.4371                \\
                                             & $2^{14}$ && 0.0117 & 0.6233 & 1.3547               & 1.4740                & 1.4913                \\
                                             & $2^{15}$ && 0.0085 & 0.6061 & 1.3925               & 1.5244                & 1.5435                \\
                                             & $2^{16}$ && 0.0062 & 0.5881 & 1.4283               & 1.5729                & 1.5940                \\
                                             & $2^{17}$ && 0.0045 & 0.5696 & 1.4620               & 1.6199                & 1.6429                \\
                                             & $2^{18}$ && 0.0033 & 0.5507 & 1.4940               & 1.6653                & 1.6905                \\
                                             & $2^{19}$ && 0.0024 & 0.5315 & 1.5244               & 1.7094                & 1.7367                \\
\hline
\end{tabular*}
\end{table}

Combining Tables~\ref{tableWC}, \ref{table4}, and \ref{values}, we get
that all obtained sample means for\break $\E\max_{i=\overline
{1,N}}B^{H}(i/N)$ satisfy the constraint
\[
\biggl(\frac{\ln(N+1)}{ N^{2H}2\pi\ln2} \biggr)^{1/2} \leq\E\max_{i=\overline
{1,N}}B^{H}(i/N)
\leq N \int_{1/2}^{1}\text{erf}^{(-1)}(2z-1)
z^{N-1}d z.
\]
Therefore, even with small values of the parameter $H$, the simulation
does not lead to contradiction.

Now let us find a lower bound for the approximation error $ \varDelta_N$.
We prove the following proposition.
\begin{proposition} Let $\varDelta_N$ be defined by \eqref{apprbounds1}.
Then, for any $H\in(0,1)$ and $N\geq1$, we have
\begin{equation}
\label{DeltLeft} \varDelta_N \geq\frac{1}{2\sqrt{H \pi e \ln2}} - N \int
_{1/2}^{1}\mathrm{erf}^{(-1)}(2z-1)
z^{N-1}d z.
\end{equation}
\end{proposition}
\begin{proof}
The statement follows from inequalities \eqref{maxbounds} and \eqref{ErB}.
\end{proof}
From this it follows that, for a fixed $N$, the approximation error
$\varDelta_N \to+ \infty$ as $ H \to0$. We also have the following
evident corollaries.

\begin{corollary}
Let $N=2^{20}$. Then
\begin{equation}
\label{upb} \varDelta_N \geq\frac{0.2055}{\sqrt{H }} - 3.4452,\quad H\in
(0,1).
\end{equation}
\end{corollary}
\begin{proof}
The statement follows from inequalities \eqref{EleftB} and \eqref{DeltLeft}.
\end{proof}

\begin{corollary}Let $N=2^{20}$. Then for the relative error, we have
\begin{equation}
\label{deltLeft} \delta_H:=\frac{\varDelta_N} { \E\max_{t \in
[0,1]}B^H(t)}\geq1- 16.765 \sqrt{H},
\quad H\in(0,1).
\end{equation}
\end{corollary}
\begin{proof}
The statement follows from inequalities \eqref{maxbounds} and \eqref{ErightB}.
\end{proof}\goodbreak

When $N=2^{20}$, from inequalities \eqref{upb} and \eqref{deltLeft} we
get the following conclusions:
\begin{itemize}
\item if $H<0.00022$, then the relative error $\delta_H\geq75\%$, and
$\varDelta_N>10.34$;
\item if $H<0.00089$, then the relative error $\delta_H\geq50\%$, and
$\varDelta_N>3.45$;
\item if $H<0.0020$, then the relative error $\delta_H\geq25\%$, and
$\varDelta_N>1.15$;
\item if $H<0.0028$, then the relative error $\delta_H\geq10\%$, and
$\varDelta_N>0.38$;
\item if $H<0.0032$, then the relative error $\delta_H\geq5\%$, and
$\varDelta_N>0.18$;
\item if $H<0.0035$, then the relative error $\delta_H\geq1\%$, and
$\varDelta_N>0.03$.
\end{itemize}

Thus, we conclude that the estimation of $ \E\max_{t \in[0,1]} B^H
(t) $ by Monte Carlo simulations leads to significant errors for small
values of the parameter $H$.

\section*{Acknowledgments}
The author is grateful to prof. Yu. Mishura for numerous interesting
discussions and active support.


\begin{thebibliography}{9}

\bibitem{Borovkov}
%
\begin{barticle}
\bauthor{\bsnm{Borovkov}, \binits{K.}},
\bauthor{\bsnm{Mishura}, \binits{Y.}},
\bauthor{\bsnm{Novikov}, \binits{A.}},
\bauthor{\bsnm{Zhitlukhin}, \binits{M.}}:
\batitle{Bounds for expected maxima of Gaussian processes and their discrete
approximations}.
\bjtitle{Stoch. Int. J. Probab. Stoch. Process.}
(\byear{2015}).
doi:\doiurl{10.1080/17442508.2015.1126282}
\end{barticle}
%
\OrigBibText
%
\begin{barticle}
\bauthor{\bsnm{Borovkov}, \binits{K.}},
\bauthor{\bsnm{Mishura}, \binits{Y.}},
\bauthor{\bsnm{Novikov}, \binits{A.}},
\bauthor{\bsnm{Zhitlukhin}, \binits{M.}}:
\batitle{Bounds for expected maxima of Gaussian processes and their discrete
approximations}.
\bjtitle{Stochastics An International Journal of Probability and Stochastic
Processes}
(\byear{2015}).
doi:\doiurl{10.1080/17442508.2015.1126282}
\end{barticle}
%
\endOrigBibText
\bptok{structpyb}%
\endbibitem

\bibitem{clark}
%
\begin{barticle}
\bauthor{\bsnm{Clark}, \binits{C.E.}}:
\batitle{The greatest of a finite set of random variables}.
\bjtitle{Oper. Res.}
\bvolume{9}(\bissue{2}),
\bfpage{145}--\blpage{162}
(\byear{1961}).
\bid{mr={0125604}}
\end{barticle}
%
\OrigBibText
%
\begin{barticle}
\bauthor{\bsnm{Clark}, \binits{C.E.}}:
\batitle{The greatest of a finite set of random variables}.
\bjtitle{Operations Research}
\bvolume{9}(\bissue{2}),
\bfpage{145}--\blpage{162}
(\byear{1961})
\end{barticle}
%
\endOrigBibText
\bptok{structpyb}%
\endbibitem

\bibitem{Coe}
%
\begin{barticle}
\bauthor{\bsnm{Coeurjolly}, \binits{J.-F.}}:
\batitle{Simulation and identification of the fractional Brownian
motion: A~bibliographical and comparative study}.
\bjtitle{J. Stat. Softw.}
\bvolume{5}(\bissue{7}),
\bfpage{1}--\blpage{53}
(\byear{2000})
\end{barticle}
%
\OrigBibText
%
\begin{barticle}
\bauthor{\bsnm{Coeurjolly}, \binits{J.-F.}}:
\batitle{Simulation and identification of the fractional Brownian
motion: A~bibliographical and comparative study}.
\bjtitle{J. Stat. Softw.}
\bvolume{5}(\bissue{7}),
\bfpage{1}--\blpage{53}
(\byear{2000})
\end{barticle}
%
\endOrigBibText
\bptok{structpyb}%
\endbibitem

\bibitem{Del}
%
\begin{barticle}
\bauthor{\bsnm{Delorme}, \binits{M.}},
\bauthor{\bsnm{Wiese}, \binits{K.J.}}:
\batitle{Maximum of a fractional Brownian motion: Analytic results from
perturbation theory}.
\bjtitle{Phys. Rev. Lett.}
\bvolume{115}(\bissue{21}),
\bfpage{210601}
(\byear{2015})
\end{barticle}
%
\OrigBibText
%
\begin{barticle}
\bauthor{\bsnm{Delorme}, \binits{M.}},
\bauthor{\bsnm{Wiese}, \binits{K.J.}}:
\batitle{Maximum of a fractional Brownian motion: Analytic results from
perturbation theory}.
\bjtitle{Physical Review Letters}
\bvolume{115}(\bissue{21}),
\bfpage{210601}
(\byear{2015})
\end{barticle}
%
\endOrigBibText
\bptok{structpyb}%
\endbibitem

\bibitem{Mol}
%
\begin{barticle}
\bauthor{\bsnm{Molchan}, \binits{G.M.}}:
\batitle{Maximum of a fractional Brownian motion: Probabilities of small
values}.
\bjtitle{Commun. Math. Phys.}
\bvolume{205}(\bissue{1}),
\bfpage{97}--\blpage{111}
(\byear{1999}).
\bid{doi={10.1007/s002200050669}, mr={1706900}}
\end{barticle}
%
\OrigBibText
%
\begin{barticle}
\bauthor{\bsnm{Molchan}, \binits{G.M.}}:
\batitle{Maximum of a fractional Brownian motion: Probabilities of small
values}.
\bjtitle{Communications in Mathematical Physics}
\bvolume{205}(\bissue{1}),
\bfpage{97}--\blpage{111}
(\byear{1999})
\end{barticle}
%
\endOrigBibText
\bptok{structpyb}%
\endbibitem

\bibitem{Sinai}
%
\begin{barticle}
\bauthor{\bsnm{Sinai}, \binits{Y.G.}}:
\batitle{Distribution of the maximum of a fractional Brownian motion}.
\bjtitle{Russ. Math. Surv.}
\bvolume{52}(\bissue{2}),
\bfpage{359}--\blpage{378}
(\byear{1997}).
\bid{doi={10.1070/RM1997v052n02ABEH001781}, mr={1480141}}
\end{barticle}
%
\OrigBibText
%
\begin{barticle}
\bauthor{\bsnm{Sinai}, \binits{Y.G.}}:
\batitle{Distribution of the maximum of a fractional Brownian motion}.
\bjtitle{Russian Mathematical Surveys}
\bvolume{52}(\bissue{2}),
\bfpage{359}--\blpage{378}
(\byear{1997})
\end{barticle}
%
\endOrigBibText
\bptok{structpyb}%
\endbibitem

\bibitem{Sudakov}
%
\begin{bbook}
\bauthor{\bsnm{Sudakov}, \binits{V.N.}}:
\bbtitle{Geometric Problems in the Theory of Infinite-Dimensional Probability
Distributions}
vol.~\bseriesno{141}.
\bpublisher{Am. Math. Soc.}
(\byear{1979}).
\bid{mr={0530375}}
\end{bbook}
%
\OrigBibText
%
\begin{bbook}
\bauthor{\bsnm{Sudakov}, \binits{V.N.}}:
\bbtitle{Geometric Problems in the Theory of Infinite-Dimensional Probability
Distributions}
vol.~\bseriesno{141}.
\bpublisher{American Mathematical Soc.}
(\byear{1979})
\end{bbook}
%
\endOrigBibText
\bptok{structpyb}%
\endbibitem

\bibitem{Tal}
%
\begin{barticle}
\bauthor{\bsnm{Talagrand}, \binits{M.}}:
\batitle{Lower classes for fractional Brownian motion}.
\bjtitle{J. Theor. Probab.}
\bvolume{9}(\bissue{1}),
\bfpage{191}--\blpage{213}
(\byear{1996}).
\bid{doi={10.1007/BF02213740}, mr={1371076}}
\end{barticle}
%
\OrigBibText
%
\begin{barticle}
\bauthor{\bsnm{Talagrand}, \binits{M.}}:
\batitle{Lower classes for fractional Brownian motion}.
\bjtitle{Journal of Theoretical Probability}
\bvolume{9}(\bissue{1}),
\bfpage{191}--\blpage{213}
(\byear{1996})
\end{barticle}
%
\endOrigBibText
\bptok{structpyb}%
\endbibitem

\bibitem{W-Ch}
%
\begin{barticle}
\bauthor{\bsnm{Wood}, \binits{A.T.}},
\bauthor{\bsnm{Chan}, \binits{G.}}:
\batitle{Simulation of stationary Gaussian processes in $[0, 1]^d$}.
\bjtitle{J. Comput. Graph. Stat.}
\bvolume{3}(\bissue{4}),
\bfpage{409}--\blpage{432}
(\byear{1994}).
\bid{doi={10.2307/1390903}, mr={1323050}}
\end{barticle}
%
\OrigBibText
%
\begin{barticle}
\bauthor{\bsnm{Wood}, \binits{A.T.}},
\bauthor{\bsnm{Chan}, \binits{G.}}:
\batitle{Simulation of stationary Gaussian processes in $[0, 1]^d$}.
\bjtitle{Journal of Computational and Graphical Statistics}
\bvolume{3}(\bissue{4}),
\bfpage{409}--\blpage{432}
(\byear{1994})
\end{barticle}
%
\endOrigBibText
\bptok{structpyb}%
\endbibitem

\end{thebibliography}
\end{document}